\documentclass[11pt, a4paper]{amsart}
\usepackage{amsfonts,amssymb,amsmath,amsthm,amscd,mathtools,multicol,tikz, tikz-cd,caption,enumerate,mathrsfs,thmtools,cite}
\usepackage{inputenc}
\usepackage[foot]{amsaddr}
\usepackage[pagebackref=true, colorlinks, linkcolor=blue, citecolor=red]{hyperref}
\usepackage{latexsym}
\usepackage{fullpage}
\usepackage{microtype}
\usepackage{subfiles} 

\usepackage{palatino}
\makeatletter
\makeindex 
\newcommand{\be}{\begin{equation}}
\newcommand{\ee}{\end{equation}}
\newcommand{\beano}{\begin{eqn*}} 
	\newcommand{\eeano}{\end{eqnarray*}}
\newcommand{\ba}{\begin{array}}
	\newcommand{\ea}{\end{array}}

\declaretheoremstyle[headfont=\normalfont]{normalhead}

\newtheorem{theorem}{Theorem}[section]
\newtheorem{theoremalph}{Theorem}[section]

\newtheorem{lemma}[theorem]{Lemma}
\newtheorem{corollary}[theorem]{Corollary}

\newtheorem{proposition}[theorem]{Proposition}

\numberwithin{equation}{section}

\def\remark{\refstepcounter{theorem}\paragraph{{\bfseries Remark~\thetheorem}}}

\begin{document}
\title{Automorphic word maps and Amit--Ashurst conjecture}

\author{Harish Kishnani}
\email{harishkishnani11@gmail.com}
\address{Indian Institute of Science Education and Research, Sector 81, Mohali 140306, India}

\author{Amit Kulshrestha}
\email{amitk@iisermohali.ac.in}
\address{Indian Institute of Science Education and Research, Sector 81, Mohali 140306, India}

\thanks{The second named author acknowledges the support from Prime Minister Research Fellowship. We are thankful to William Cocke whose survey talk introduced us to Amit conjecture, and to Josu Sangroniz for e-mail correspondence.}
\subjclass[2020]{20D15, 20F10}
\keywords{word maps, finite nilpotent groups, Amit--Ashurst conjecture}

\begin{abstract}
In this article, we address Amit--Ashurst conjecture on lower bounds of a probability distribution associated to a word on a finite nilpotent group. We obtain an extension of a result of \cite{CIT_2020} by providing improved bounds for the case of finite nilpotent groups of class $2$ for words in an arbitrary number of variables, and also settle the conjecture in some cases. We achieve this by obtaining words that are identically distributed on a group to a given word. In doing so, we also obtain an improvement of a result of \cite{IS_2017} using elementary techniques.
\end{abstract}
\maketitle

\section{Introduction}
Let $G$ be a finite group and $F_k$ be the free group in free variables $x_1, x_2, \dots, x_k$. Elements of $F_k$ are called \emph{words}. The evaluation of a word $w \in F_k$ on $k$-tuples $(g_1, g_2, \cdots, g_k) \in G^k$ induces a map $G^k \to G$. It is called the \emph{word map} on $G$ induced by $w$.
The notation $w$ is used for both, an element of a free group and the induced word map $w : G^k \to G$. The collection of word maps on $G$ forms a group $F_k(G)$ under pointwise operation. Thus, we have a homomorphism $F_k \to F_k(G)$ that maps $w$ to the corresponding word map on $G$.

The image of a word map $w$ is denoted by $w(G)$. For $g \in G$, the \emph{fiber} of $w$ at $g$ is the subset
$$w^{-1}(g) := \{(g_1, g_2, \dots, g_k) \in G^k: w(g_1, g_2, \dots, g_k) = g\}$$
of $G^k$. We denote $P_{w,G}(g) := |w^{-1}(g)|/|G|^k$. It is evident that
if $g \notin w(G)$, then $P_{w,G}(g) = 0$.
The map $P_{w,G} : G \to [0,1]$ given by $g \mapsto P_{w,G}(g)$ is a probability function on $G$ in the sense that $\displaystyle \sum_{g \in G} P_{w,G}(g) = 1$. 

The function $P_{w,G}$ has been studied in \cite{Levy_2011}, \cite{IS_2017}, \cite{CIT_2020}, \cite{Cocke-Nilpotent}, \cite{CCT_2022} and \cite{Cocke-Ho_19}. Some of these studies are inspired by the conjectures of Amit and Ashurst concerning a lower bound on $P_{w,G}$ over finite nilpotent groups. The conjecture of Amit states that for finite nilpotent groups $G$ and words $w \in F_k$, the number $P_{w,G}(1)$ is bounded below by $|G|^{-1}$. While the conjecture is still open, a stronger version of the conjecture was proposed by Ashurst in her thesis as a question.
Ashurst asks whether $|G|^{-1}$ is a lower bound for
$P_{w,G}(g)$ for every $g \in G$, where $G$ is a finte nilpotent group. Following \cite{CCT_2022}, we refer to the question of Ashurst as Amit--Ashurst conjecture.
We redirect the reader to \cite{CIT_2020} for details on Amit--Ashurst conjecture. Toward this conjecture, it was shown in \cite{CIT_2020},
that if $p$ is an odd prime and $G$ is a finite $p$-group of nilpotency class $2$, then for each $w \in F_k$, and $g \in G$, the bound $P_{w,G}(g) \geq |G|^{-2}$ holds. We obtain the following improvement of it.

\begin{theoremalph}(Theorem \ref{Improved_bound})
Let $G$ be a finite nilpotent group of class $2$ and $w \in F_k$.
Then for each $g \in G$, we have $P_{w,G}(g) \geq |G'|^{-1}|G|^{-1}$, where $G'$ denotes the derived subgroup of $G$. 
\end{theoremalph}
We achieve this using elementary methods. In the course of it, we prove the following theorem.

\begin{theoremalph}(Theorem \ref{word_as_chain})
Let $p$ be a prime and $\mathcal P$ be the set consisting of $0, 1$ and positive integral powers of $p$.
Let $G$ be a $p$-group of class $2$ and $w \in F_k$. Then there exist $t_0, t_1, t_2, \cdots, t_{k-1} \in \mathcal P$ such that $w$ is $F_k(G)$-automorphic to the word $x_1^{t_0} [x_1,x_2]^{t_1} [x_2,x_3]^{t_2} \cdots [x_{k-1},x_k]^{t_{k-1}}$.
\end{theoremalph}

One of the byproducts of this result is an elementary proof of \cite[Corollary 4.3]{CIT_2020} which asserts that every finite group of nilpotency class $2$ is rational. 

In this paper, we also improve \cite[Proposition 2.3]{IS_2017}, to Theorem \ref{TheoremC}, again using elementary techniques.

For $t_0 \geq 0$, $1 \leq \ell \leq n \leq k$, and $s_i \in \mathbb N \cup \{0\}$, we denote
$$v_p(t_0,\ell; s_1,s_2,\cdots,s_{n-1}) := x_1^{t_0} \left( \prod_{i=1}^{\ell-1} [x_i,x_{i+1}]^{p^{s_{i}}} \right) \left( [x_\ell,x_{\ell+1}]^{p^{s_{\ell}}} [x_{\ell+2},x_{\ell+3}]^{p^{s_{\ell+2}}} \dots [x_{n-1},x_n]^{p^{s_{n-1}}} \right)$$
The word $v_p(t_0,\ell; s_1,s_2,\cdots,s_{n-1}) \in F_n$ has three parts. We refer to $x_1^{t_0}$ as the power part,
$\prod_{i=1}^{\ell-1} [x_i,x_{i+1}]^{p^{s_{i}}}$ as the
linked commutator part, and
$[x_\ell,x_{\ell+1}]^{p^{s_{\ell}}} [x_{\ell+2},x_{\ell+3}]^{p^{s_{\ell+2}}} \dots [x_{n-1},x_n]^{p^{s_{n-1}}}$ as the disjoint commutator part.

\begin{theoremalph}\label{TheoremC}(Theorem \ref{exhaustive_p})
Let $w \in F_k$ be a word and $G$ be a $p$-group of nilpotency class $2$. 
\begin{enumerate}
\item[(i).] If $w \notin [F_k, F_k]$ then either $P_{w,G}$ is a constant function or $w$ and $v(p^r, \ell; s_1,\cdots,s_{n-1})$ are $F_k(G)$-automorphic for some $r > 0$; 
$s_{1} < s_{2} < \dots < s_{\ell-1} < r$, and
$s_j < r$ for $j \geq \ell$.

\item[(ii).] If $w \in [F_k, F_k]$ then for some $s_1, s_2, \dots, s_{n-1}$, the words $w$ and $v(0, 1; s_1,s_2,\cdots,s_{n-1})$ are $F_k(G)$-automorphic.
\end{enumerate}
\end{theoremalph}

We make use of these theorems for settling the Amit--Ashurst conjecture in specific cases (Corollary \ref{condition_on_G-prime}, Corollary \ref{extraspecial-p_groups}, Corollary \ref{power_word}, Corollary \ref{surjective_word}). Finally, we conclude the article by recording that a word $w \in F_k$ is surjective on a finite nilpotent group $G$ of class $2$ if and only if $w$ is $F_k(G)$-automorphic to a primitive word. These conditions are equivalent to the condition that $P_{w,G}$ is a uniform probability distribution. 
 
 \section{Automorphisms of word maps}
In this section, we obtain some results which are to be used later in this article. These results concern identifying suitable representatives of words up to $F_k(G)$-automorphism.

Two words $w_1$ and $w_2$ are said to be \emph{identically distributed on} $G$ if $P_{w_1,G} = P_{w_2,G}$, as set theoretic functions on $G$. We do not impose identically distributed words to have the same number of variables. For instance, $w$ and $x_1$ are identically distributed if and only if $P_{w,G}$ is a constant function.

Let $Aut(F_k)$ denote the group of automorphisms of $F_k$ and $\varphi \in Aut(F_k)$. Then $\varphi$ is said to be an \emph{elementary Nielsen transformation} if one of the following holds.
\begin{enumerate}
\item[(i).] $\varphi$ swaps $x_1$ and $x_2$, and fixes all other $x_i$'s.
\item[(ii).] $\varphi$ cyclically permutes $x_1, x_2, \dots , x_k$ to $x_2, \dots , x_k, x_1$.
\item[(iii).] $\varphi$ maps $x_1$ to ${x_1}^{-1}$.
\item[(iv).] $\varphi$ maps $x_1$ to $x_1x_2$.
\end{enumerate}

It is evident that if $\varphi$ is an elementary Nielsen transformation then $w$ and $\varphi(w)$ are identically distributed. Since elementary Nielsen transformations generate $Aut(F_k)$ (see \cite{Nielsen-24}), we have the following.

\begin{proposition}\label{aut_eq}
Let $G$ be a finite group and $\varphi \in Aut(F_k)$. Then for every $w \in F_k$, the words $w$ and $\varphi(w)$ are identically distributed on $G$.
\end{proposition}

\remark \label{F_k(G)-automorphic} Let $G$ be a group and $w_1, w_2 \in F_k$. Asking $w_1$ and $w_2$ to be automorphic in $F_k$ in order to establish that these words are identically distributed on $G$ is bit too demanding. A milder condition to ask for is the automorphism of $w_1$ and $w_2$ as elements of the group $F_k(G)$. When this happens, we call
the two words to be $F_k(G)$-\emph{automorphic}.

\remark \label{class-2_eq} Let $G$ be a finite group of nilpotency class $2$ and $w \in F_k$. Then for every $w' \in [[F_k,F_k],F_k]$, the words $w$ and $ww'$ induce the same word map on $G$, where $ww'$ is the juxtaposition of the words $w$ and $w'$. This holds because for every $(g_1, g_2, \dots, g_k) \in G^k$ and $w' \in [[F_k,F_k],F_k]$, we have
$w'(g_1, g_2, \dots, g_k) = 1$. 

This remark will be used often in conjunction with Proposition \ref{aut_eq} throughout the paper.
We also record the following well-known identities for further use.

\begin{lemma}\label{powers-of-commutators}
Let $G$ be a nilpotent group of class $2$. For $g,h \in G$, denote $[g,h] := ghg^{-1}h^{-1}$. Let $n \in \mathbb Z$. Then,
\begin{enumerate}
\item[(i).] $[g,h]^n = [g^n,h] = [g,h^n]$.
\item[(ii).] $[g^ih^j,g^kh^l]=[g,h]^{il-jk}, \forall g,h\in G$.
\item[(iii).] $(gh)^n=g^n h^n [h,g]^{\frac{n(n-1)}{2}}$.
\end{enumerate}
\end{lemma}

The next lemma is an extension of  \cite[Lemma 3.5]{KKK_2022}. We introduce some notation before stating it. For $r$-tuples $I_1 = (i_{11}, i_{21}, \cdots, i_{r1}), I_2 = (i_{12}, i_{22}, \cdots, i_{r2}), \cdots, I_k = (i_{1k}, i_{2k}, \cdots, i_{rk}) \in \mathbb Z^r$ we denote,
\begin{align*}
|I_t| &:= i_{1t} + i_{2t} + \cdots + i_{rt}; \quad 1 \leq t \leq k \\
w_{I_1,I_2,\dots,I_k} &:= (x_1^{i_{11}}x_2^{i_{12}}\dots x_k^{i_{1k}}) \dots (x_1^{i_{r1}}x_2^{i_{r2}}\dots x_k^{i_{rk}}) \in F_k\\
\end{align*}

\begin{lemma}\label{commutator_products}
Let $I_1, I_2, \dots I_k \in \mathbb Z^r$, be such that $|I_1| = |I_2| = \cdots = |I_k| =0$. Let $G$ be a nilpotent group of class $2$. 
Then, there exist $\beta_{m,n} \in \mathbb{Z}$ such that $w_{I_1,I_2,\dots,I_k}$ and $\displaystyle \prod_{m=1}^{k-1} \prod_{n = m+1}^{k} [x_m,x_n]^{\beta_{m,n}}$ induce the same word map on $G$.
\end{lemma}

\begin{proof}
Let $G$ be a nilpotent group of class $2$. We use induction on $k$ to prove this lemma. The case $k=2$ follows from \cite[Lemma 3.5]{KKK_2022}. Let us assume the lemma for words in $F_{k-1}$. In
$$ w_{I_1,I_2,\dots,I_k}=(x_1^{i_{11}}x_2^{i_{12}} \cdots x_k^{i_{1k}})(x_1^{i_{21}}x_2^{i_{22}} \cdots x_k^{i_{2k}}) \cdots (x_1^{i_{r1}}x_2^{i_{r2}} \cdots x_k^{i_{rk}}) \in F_k,$$
we substitute $i_{rk}=-(i_{1k}+i_{2k}+\cdots +i_{(r-1)k})$ so as to obtain the following equalities of word maps on $G$.
\begin{align*}
w_{I_1,I_2,\dots,I_k} &=(x_1^{i_{11}}x_2^{i_{12}} \cdots x_k^{i_{1k}}) \cdots (x_1^{i_{(r-1)1}} \cdots x_k^{i_{(r-1)k}})(x_1^{i_{r1}} \cdots x_k^{-(i_{1k}+i_{2k}+\cdots +i_{(r-1)k})}) \\
 &= (x_1^{i_{11}}x_2^{i_{12}} \cdots x_k^{i_{1k}}) \cdots (x_1^{i_{(r-1)1}} \cdots x_k^{i_{(r-1)k}})(x_1^{i_{r1}} \cdots x_k^{-i_{(r-1)k}}x_k^{-(i_{1k}+i_{2k}+\cdots +i_{(r-2)k})}) \\ 
 &= (x_1^{i_{11}}x_2^{i_{12}} \cdots x_k^{i_{1k}}) \cdots (x_1^{i_{(r-1)1}} \cdots x_k^{i_{(r-1)k}})w' x_k^{-i_{(r-1)k}} {(w')}^{-1}w' x_k^{-(i_{1k}+i_{2k}+\cdots +i_{(r-2)k})},
\end{align*}
where $w'=x_1^{i_{r1}}x_2^{i_{r2}} \cdots x_{k-1}^{i_{r(k-1)}} \in F_{k-1}$.
Thus, as word maps on $G$
\begin{align*}
w_{I_1,I_2,\dots,I_k} &= (x_1^{i_{11}}x_2^{i_{12}} \cdots x_k^{i_{1k}}) \cdots (x_1^{i_{(r-1)1}} \cdots x_{k-1}^{i_{(r-1)(k-1)}}) [x_k^{i_{(r-1)k}}, w']w' x_k^{-(i_{1k}+i_{2k}+\cdots +i_{(r-2)k})}.
\end{align*}
By Remark \ref{class-2_eq}, 
\begin{align*}
w_{I_1,I_2,\dots,I_k} &= (x_1^{i_{11}}x_2^{i_{12}} \cdots x_k^{i_{1k}}) \cdots (x_1^{i_{(r-1)1}} \cdots x_{k-1}^{i_{(r-1)(k-1)}}) w' x_k^{-(i_{1k}+i_{2k}+\cdots + i_{(r-2)k})}[x_k^{i_{(r-1)k}}, w'].
\end{align*}
We now use Lemma \ref{powers-of-commutators} and obtain
\begin{align*}
w_{I_1,I_2,\dots,I_k} &= w_1 \prod_{m=1}^{k-1} [x_m,x_k]^{-i_{rm}i_{(r-1)k}}.
\end{align*}
where
$w_1 := (x_1^{i_{11}}x_2^{i_{12}} \cdots x_k^{i_{1k}}) \cdots (x_1^{i_{(r-1)1}} \cdots x_{k-1}^{i_{(r-1){k-1}}}) w' x_k^{-(i_{1k}+i_{2k}+\cdots + i_{(r-2)k})}$.
Observe that $w_1 \in F_k$ satisfies the hypothesis of the lemma.
Thus, repeating steps as above and thereby pushing the powers of $x_k$ in the commutator, we eventually have a word $w_{r} \in F_{k-1}$ 
such that
\begin{align*}
w_{I_1,I_2,\dots,I_k} &= w_r \prod_{m=1}^{k-1} [x_m,x_k]^{\beta_{m,k}},
\end{align*}
as word maps on $G$, for some $\beta_{m,k} \in \mathbb{Z}$. Observe that $w_r$ satisfies the hypothesis of the Lemma in $k-1$ free variables, hence we are done.
\end{proof}

The next lemma also concerns words that are products of commutators.

\begin{lemma}\label{disjoint_commutators}
Let $n \leq k$ and $\alpha \in \mathbb N$. Let $w =[x_1,x_2]^{\alpha^{s_1}} [x_2,x_3]^{\alpha^{s_2}} \cdots [x_{n-1},x_n]^{\alpha^{s_{n-1}}}$ be a word in $F_k$. Let $n' = 2 \lfloor n/2 \rfloor$. Then, there exist nonnegative integers $t_1, t_3, \cdots, t_{n'-1}$ such that on any nilpotent group $G$ of nilpotency class $2$, the word $w$ is $F_k(G)$-automorphic to  
$[x_1,x_2]^{t_1} [x_3,x_4]^{t_3} \cdots [x_{n'-1},x_{n'}]^{t_{n'-1}}$, where $t_i \in \{0, 1, \alpha^{r_i}\}$ for some positive integers $r_i$.
\end{lemma}

\begin{proof}
We use induction on $n$ to prove this lemma. For $n=2$, the word $w$ is a power of a single commutator, and there is nothing to prove. So, we suppose that $n \geq 3$ and assume that the lemma holds for all $m \leq n-1$. 
Now, we consider the two cases.
\begin{enumerate}
\item[{\bfseries Case 1.}] $s_{n-1} \leq s_{n-2}$ \quad
Under the transformation $x_n \mapsto x_n ({x_{n-2}})^{\alpha^{s_{n-2}-s_{n-1}}}$, the word $w$ maps to
$[x_1,x_2]^{\alpha^{s_{1}}} [x_2,x_3]^{\alpha^{s_{2}}} \cdots [x_{n-3},x_{n-2}]^{\alpha^{s_{n-3}}} [x_{n-1},x_n]^{\alpha^{s_{n-1}}}$. The last two terms of this word are powers of disjoint commutators. Now, by induction, assuming the lemma for the word $[x_1,x_2]^{\alpha^{s_{1}}} [x_2,x_3]^{\alpha^{s_{2}}} \cdots [x_{n-3},x_{n-2}]^{\alpha^{s_{n-3}}}$, we are done. \\

\item[{\bfseries Case 2.}] $s_{n-1} > s_{n-2}$ \quad Let $\ell$ be the smallest index such that $s_{\ell} < \cdots s_{n-2}< s_{n-1}$. First, suppose that $\ell=1$.
The automorphic transformation $x_1 \mapsto x_1{x_{3}}^{\alpha^{s_{2}-s_{1}}}$ maps $w$ to the word
$[x_1,x_2]^{\alpha^{s_{1}}} [x_3,x_4]^{\alpha^{s_{3}}} [x_4,x_5]^{\alpha^{s_{4}}} \dots [x_{n-1},x_n]^{\alpha^{s_{n-1}}}$. The first two terms of this word are powers of disjoint commutators. Now, we appeal to the induction on the word $[x_3,x_4]^{\alpha^{s_{3}}} [x_4,x_5]^{\alpha^{s_{4}}} \dots [x_{n-1},x_n]^{\alpha^{s_{n-1}}}$ to conclude the lemma. \\

Next, suppose that $\ell \geq 2$. Then, we have $s_{\ell -1} \geq s_{\ell} < s_{\ell +1} < \cdots < s_{n-2} < s_{n-1}$. Under the automorphic transformation $x_\ell \mapsto x_\ell ({x_{\ell+2}})^{\alpha^{s_{\ell+1}-s_{\ell}}}$, the word $w$ maps to the word

$$[x_1,x_2]^{\alpha^{s_{1}}} \dots [x_{\ell-1},x_\ell]^{\alpha^{s_{\ell-1}}} [x_{\ell-1},x_{\ell+2}]^{\alpha^{s_{\ell-1}+(s_{\ell+1}-s_{\ell})}} [x_{\ell},x_{\ell+1}]^{\alpha^{s_{\ell}}} [x_{\ell+2},x_{\ell+3}]^{\alpha^{s_{\ell+2}}} \dots [x_{n-1},x_n]^{\alpha^{s_{n-1}}}.$$

\noindent On this word, we further apply the transformation $x_{\ell+1} \mapsto x_{\ell+1} ({x_{\ell-1}})^{\alpha^{s_{\ell-1}-s_{\ell}}}$ to obtain the following word which is $F_k(G)$-automorphic to $w$.

$$[x_1,x_2]^{\alpha^{s_{1}}} \dots [x_{\ell-2},x_{\ell-1}]^{\alpha^{s_{\ell-2}}} [x_{\ell-1},x_{\ell+2}]^{\alpha^{s_{\ell-1}+(s_{\ell+1}-s_{\ell})}} [x_{\ell},x_{\ell+1}]^{\alpha^{s_{\ell}}} [x_{\ell+2},x_{\ell+3}]^{\alpha^{s_{\ell+2}}} \dots [x_{n-1},x_n]^{\alpha^{s_{n-1}}}.$$

\noindent Since $G$ is a nilpotent group of class 2, this word is $F_k(G)$-automorphic to the following word obtained after permuting commutators.
$$[x_1,x_2]^{\alpha^{s_{1}}} \dots [x_{\ell-2},x_{\ell-1}]^{\alpha^{s_{\ell-2}}} [x_{\ell-1},x_{\ell+2}]^{\alpha^{s_{\ell-1}+(s_{\ell+1}-s_{\ell})}} [x_{\ell+2},x_{\ell+3}]^{\alpha^{s_{\ell+2}}} \dots [x_{n-1},x_n]^{\alpha^{s_{n-1}}}[x_{\ell},x_{\ell+1}]^{\alpha^{s_{\ell}}}.$$

\noindent Now, the proof follows by induction on the word
$$[x_1,x_2]^{\alpha^{s_{1}}} \dots [x_{\ell-2},x_{\ell-1}]^{\alpha^{s_{\ell-2}}} [x_{\ell-1},x_{\ell+2}]^{\alpha^{s_{\ell-1}+(s_{\ell+1}-s_{\ell})}} [x_{\ell+2},x_{\ell+3}]^{\alpha^{s_{\ell+2}}} \dots [x_{n-1},x_n]^{\alpha^{s_{n-1}}}.$$

\end{enumerate}
\end{proof}

\begin{theorem}\label{word_as_chain}
Let $p$ be a prime and $\mathcal P$ be the set consisting of $0, 1$ and positive integral powers of $p$.
Let $G$ be a $p$-group of class $2$ and $w \in F_k$. Then there exist $t_0, t_1, t_2, \cdots, t_{k-1} \in \mathcal P$ such that $w$ is $F_k(G)$-automorphic to the word $x_1^{t_0} [x_1,x_2]^{t_1} [x_2,x_3]^{t_2} \cdots [x_{k-1},x_k]^{t_{k-1}}$.
\end{theorem}

\begin{proof}
By \cite[Lemma 2.3]{CockeHoChirality}, we assume that  $w = x_1^t c$ for some $t \in \mathbb Z$ and $c \in [F_k,F_k]$. Now, by Lemma \ref{commutator_products}, the word $w$ is $F_k(G)$-automorphic to a word of type
$$x_1^{t} \displaystyle \prod_{m=1}^{k-1} \prod_{n=m+1}^{k} [x_m,x_n]^{\beta_{m,n}},$$
for some $t,\beta_{m,n} \in \mathbb{Z}$. We bunch together the commutators involving $x_1$ to obtain
that $w$ is $F_k(G)$-automorphic to
$$v := x_1^{t} \displaystyle \prod_{n=2}^{k} [x_1,x_n]^{\beta_{1,n}} \prod_{m=2}^{k-1} \prod_{n=m+1}^{k} [x_m,x_n]^{{\beta_{m,n}}}.$$ We substitute $t = p^a \alpha$ and $\beta_{m,n} = p^{b_{m,n}}\alpha_{m,n}$, where  $\alpha$ and $\alpha_{m,n}$ are integers which are either $0$ or coprime to $p$. Since a permutation of variables is an automorphism of $F_k$, after a suitable permutation of $x_1, x_2, \cdots, x_k$,  we assume that $\alpha_{1,n} = 0$ whenever $n$ exceeds a threshold, say $u_1$. If required, we further permute these free variables so as to ensure that $b_{1,2} \leq b_{1,3} \leq \cdots \leq b_{1,u_1}$.  If $u_1\geq 3$ then under the automorphic transformation $x_2 \mapsto x_2x_3^{-{(\alpha_{1,2})}^{-1}\alpha_{1,3}p^{(b_{1,3}-b_{1,2)}}}$, the word $[x_1,x_2]^{{p^{b_{1,2}}}\alpha_{1,2}}[x_1,x_3]^{{p^{b_{1,3}}}\alpha_{1,3}}$ maps to the word
$[x_1,x_2]^{{p^{b_{1,2}}}\alpha_{1,2}}$. Here $\alpha_{1,2}^{-1} \in \mathbb Z$ is an inverse of $\alpha_{1,2}$ mod $p$. Thus, $v$ is $F_k(G)$-automorphic to a word of type
$$x_1^{p^a\alpha} \displaystyle [x_1,x_2]^{{p^{b_{1,2}}}\alpha_{1,2}}  \prod_{n=4}^{u_1} [x_1,x_n]^{p^{b_{1,n}}\alpha_{1,n}} \prod_{m=2}^{k-1} \prod_{n=m+1}^{k} [x_m,x_n]^{p^{b'_{m,n}}\alpha'_{m,n}},$$
where $b'_{m,n} , \alpha'_{m,n}$ are some integers. We notice that the commutator $[x_1, x_3]$ does not show up in this word. In fact, a repetition of  transformations akin to the one above, allows us to eliminate commutators of the form $[x_1,x_n]^{{p^{b_{1,n}}}\alpha_{1,n}}$, whenever $n \geq 3$, and thereby we construct a sequence of words which are $F_k(G)$-automorphic to $v$.

Now, we assume for $m=2$, that after suitable permutations, there exists $u_2 \in \mathbb{N}$ such that $b_{2,3} \leq b_{2,4} \leq \cdots \leq b_{2,u_2}$, and $b_{2,n}=0$ for all $n > u_2$. Following an argument as in the above paragraph, we eliminate commutators $[x_2,x_n]^{{p^{b_{2,n}}}\alpha_{2,n}}$ for $n \geq 4$ ,  and construct a word which is $F_k(G)$-automorphic to $v$.  We continue this elimination process for $m=3,4,\dots, k-2$ and obtain that $v$ is $F_k(G)$-automorphic to a word of the form
$w' = x_1^{p^{s_0}\gamma_0} [x_1,x_2]^{{p^{s_{1}}}\gamma_{1}} [x_2,x_3]^{{p^{s_{2}}}\gamma_{2}} \cdots [x_{k-1},x_k]^{{p^{s_{k-1}}}\gamma_{k-1}},$
where ${s_{i}} \in \mathbb{N}\cup \{0\}$, and ${\gamma_{i}}$ are either $0$ or coprime to $p$. 

Define, for $i \in \{0, 1, \cdots, k-1\}$,
$$t_i = 
\begin{cases}
0 \quad \text{if } \gamma_i = 0 \\
p^{s_i} \quad \text{if } \gamma_i \neq 0
\end{cases}
$$
and
$$\epsilon_i = 
\begin{cases}
1 \quad \text{if } \gamma_i = 0 \\
\gamma_i \quad \text{if } \gamma_i \neq 0
\end{cases}
$$
Let $w'' = x_1^{t_0} [x_1,x_2]^{t_1} [x_2,x_3]^{t_2} \cdots [x_{k-1},x_k]^{t_{k-1}}$. We show that $w'$ and $w''$ are $F_k(G)$-automorphic. Observe that for each $\gamma$ that is coprime to $p$, the map $x_i \mapsto x_i^{\gamma}$ is an $F_k(G)$- automorphism. We consider the automorphism $\psi: F_k(G) \mapsto F_k(G)$ that maps each $x_i$ to $x_i^{\delta_{i-1}}$, 
where $\delta_0 = \epsilon_0$, and $\delta_i$ for $i \geq 1$ are determined by the congruences $\delta_i{\delta_{i-1}} \equiv \epsilon_{i} \mod p$. Then, $\psi(w'') = w'$.
Since $w$ is $F_k(G)$-automorphic to $v$, which is $F_k(G)$-automorphic to $w'$, and $w'$ is $F_k(G)$-automorphic to $w''$, the words $w$ and $w''$ are $F_k(G)$-automorphic. This completes the proof.
\end{proof}




A consequence of Theorem \ref{word_as_chain} concerns the Amit--Ashurst conjecture. We  discuss it in the next section.

\section{Toward Amit--Ashurst conjecture}

\begin{theorem}\label{Improved_bound}
Let $G$ be a finite nilpotent group of class $2$ and $w \in F_k$.
Then for each $g \in G$, we have $P_{w,G}(g) \geq |G'|^{-1}|G|^{-1}$, where $G'$ denotes the derived subgroup of $G$. 
\end{theorem}

\begin{proof}
It is enough to assume that $G$ is a $p$-group for some prime $p$. Then, by Theorem \ref{word_as_chain}, there exist $t_0, t_1, t_2, \cdots, t_{k-1} \in \mathcal P$, where the set $\mathcal P$ consists of $0, 1$ and positive integral powers of $p$, such that $w$ is $F_k(G)$-automorphic to the word $x_1^{t_0} [x_1,x_2]^{t_1} [x_2,x_3]^{t_2} \cdots [x_{k-1},x_k]^{t_{k-1}}$. 

Let $g \in w(G)$ and 
$g_1,g_2,\dots,g_k \in G$ be such that $w(g_1,g_2,\dots,g_k) = g$. 
Let $k_1=\lceil \frac{k}{2} \rceil$ and $k_2=\lfloor \frac{k}{2} \rfloor$.
Define the maps $\varphi_{odd} : G' \times G^{k_1-1} \to G'$
and $\varphi_{even} : G^{k_2} \to G'$
as follows
\begin{align*}
\varphi_{odd}(y_1,y_3,\dots,y_{2k_1 -1})&=
\begin{cases}
w(y_1,g_2,y_3,g_4,\dots,y_{k-1},g_k), \text{ if } k \text{ is even} \\
w(y_1,g_2,y_3,g_4,\dots,g_{k-1},y_k), \text{ if } k \text{ is odd} 
\end{cases}  \\
\varphi_{even}(y_2,y_4,\dots,y_{2k_2}) &= 
\begin{cases}
{g_1}^{-t_0}w(g_1,y_2,g_3,y_4,\dots,g_{k-1},y_k), \text{ if } k \text{ is even} \\
{g_1}^{-t_0}w(g_1,y_2,g_3,y_4,\dots,y_{k-1},g_k), \text{ if } k \text{ is odd} 
\end{cases}
\end{align*}
It is easy to see that for finite nilpotent groups of class $2$, these maps are homomorphisms. 

Let $(h_1, h_3, \cdots, h_{2k_1-1}) \in {\rm ker}(\varphi_{odd})$ and $(h_2,h_4,\dots,h_{2k_2}) \in {\rm ker}(\varphi_{even})$. We make the following computations in two cases.\\
{\bfseries Case 1. } $k$ is odd.
\begin{align*}
w(h_1g_1,g_2,h_3g_3\dots,g_{k-1},h_{k}g_{k})&=\varphi_{odd}(h_1 g_1,h_3 g_3,\dots,h_{k} g_{k})\\
&=\varphi_{odd}(g_1,g_3,\dots,g_{k})\\&=w(g_1,g_2,g_3,g_4,\dots,g_{k-1},g_k)=g \\
\end{align*}
\begin{align*}
w(g_1,h_2 g_2,\dots,h_{k-1}g_{k-1},g_k)&={g_1}^{t_0}\varphi_{even}(h_2 g_2,h_4 g_4,\dots,h_{k-1} g_{k-1})\\
&={g_1}^{t_0}\varphi_{even}(g_2,g_4,\dots,g_{k-1}) \\ &=w(g_1,g_2,g_3,g_4,\dots,g_{k-1},g_k)=g
\end{align*}

{\bfseries Case 2. } $k$ is even.
\begin{align*}
w(h_1g_1,g_2,h_3g_3\dots,h_{k-1}g_{k-1},g_k)&=\varphi_{odd}(h_1 g_1,h_3 g_3,\dots,h_{k-1} g_{k-1})\\
&=\varphi_{odd}(g_1,g_3,\dots,g_{k-1})\\&=w(g_1,g_2,g_3,g_4,\dots,g_{k-1},g_k)=g
\end{align*}

\begin{align*}
w(g_1,h_2 g_2,\dots,g_{k-1},h_k g_k)&={g_1}^{t_0}\varphi_{even}(h_2 g_2,h_4 g_4,\dots,h_k g_k)\\
&={g_1}^{t_0}\varphi_{even}(g_2,g_4,\dots,g_k) \\ &=w(g_1,g_2,g_3,g_4,\dots,g_{k-1},g_k)=g
\end{align*}

Thus, if $w(g_1, g_2, \cdots, g_k) = g \in w(G)$, then there are at least $|{\rm ker}(\varphi_{odd})|.|{\rm ker}(\phi_{even})|$ many elements in the fiber $w^{-1}(g)$.
Consequently, $$P_{w,G}(g) = \frac{|w^{-1}(g)|}{|G^k|}\geq 
\frac{1}{|G^k|}|{\rm ker}(\varphi_{odd})|.|{\rm ker}(\varphi_{even})| \geq
\frac{1}{|G^k|}
|G|^{\lceil \frac{k}{2} \rceil-1}.
\frac{|G|^{\lfloor \frac{k}{2}\rfloor}}{|G'|} =\frac{1}{|G'||G|}.$$ \end{proof}

In the following corollary, we prove Amit--Ashurst conjecture in a particular case.

\begin{corollary}\label{condition_on_G-prime}
Let $p$ be an odd prime and $G$ be a finite $p$-group of nilpotency class $2$ with the property that $g^p \in G'$ for every $g \in G $. Let $w \in F_k$ and $g \in G$. Then, 
$$P_{w,G}(g) \geq \frac{1}{|G'|^2}$$
Thus, if $|\frac{G}{G'}| \geq |G'|$ then Amit--Ashurst conjecture holds for $G$.
\end{corollary}

\begin{proof}
The proof is based on the main idea of Theorem \ref{Improved_bound}.
Let $g \in w(G)$ and $(g_1, g_2, \cdots, g_k)$ be such that $w(g_1, g_2, \cdots, g_k)=g$. For $r = \lceil \frac{k}{2}\rceil$ consider the map $\psi_{odd} : G^r \to G'$ defined by
$$\psi_{odd}(y_1,y_3,\dots,y_{2r-1})=
\begin{cases}
w(y_1,g_2,y_3,g_4,\dots,y_{k-1},g_k), \text{ if } k \text{ is even} \\
w(y_1,g_2,y_3,g_4,\dots,g_{k-1},y_k), \text{ if } k \text{ is odd} 
\end{cases}
$$
The map $\psi_{odd}$ may easily be checked to be a homomorphism.
The hypothesis that $p$ is odd is to be used in this checking.
Now, if $w(g_1, g_2, \cdots, g_k) = g \in w(G)$, then there are at least $|{\rm ker}(\psi_{odd})|.|{\rm ker}(\varphi_{even})|$ many elements in the fiber $w^{-1}(g)$, where 
$\varphi_{even}$ is the homomorphism defined in  the proof of Theorem \ref{Improved_bound}.

Consequently, $$P_{w,G}(g) = \frac{1}{|G^k|}N_{(G,w=g)} \geq 
\frac{1}{|G^k|}|{\rm ker}(\psi_{odd})|.|{\rm ker}(\varphi_{even})| \geq
\frac{1}{|G^k|}
\frac{|G|^{\lceil \frac{k}{2}\rceil}}{|G'|}
\frac{|G|^{\lfloor \frac{k}{2}\rfloor}}{|G'|} =\frac{1}{|G'|^2}.$$  
\end{proof}

Let $\Phi(G)$ denote the Frattini subgroup of a group $G$.
We recall that a $p$-group $G$ is called a \emph{special $p$-group} if $G' = \Phi(G) = Z(G)$, and both $Z(G)$ and $G/Z(G)$ are
elementary abelian $p$-groups. Additionally, if $Z(G)$ is the cyclic $p$-group then $G$ is called an
\emph{extraspecial $p$-group}. We remark that extraspecial $p$-groups satisfy the hypothesis of the above corollary, and hence we have the following.

\begin{corollary}\label{extraspecial-p_groups}
Let $p$ be an odd prime and $G$ be an extraspecial $p$-group. Then Amit--Ashurst conjecture holds for $G$.
\end{corollary}

Our methods do not establish Amit--Ashurst conjecture for special $p$-groups. This is because there exist special $p$-groups $G$ such that $|\frac{G}{G'}| < |G'|$. One such example has the following presentation.

\begin{align*}
G = \langle e_1, \cdots, e_4, f_1, \cdots,  f_5 ~|~ & e_i^p = f_j^p = 1, [e_i, f_j] = 1, [f_j, f_k] = 1,\\ &[e_1, e_2] = f_1, [e_1, e_3] = f_2, [e_1, e_4] = f_3, \\ 
& [e_2, e_3] = f_4, [e_2, e_4] = f_5, [e_3, e_4] = 1; \\ & \quad \quad \quad \quad \quad \quad \quad \quad 1 \leq i \leq 4; 1 \leq j,k \leq 5 \rangle
\end{align*}

The group $G$ defined above has order $p^9$. This is the smallest order where special $p$-groups do not satisfy $|\frac{G}{G'}| \geq |G'|$. We refer to \cite[Remark 2.4]{Renza-Cortini} for details. We are not aware of any result in literature that proves Amit--Ashurst conjecture for special $p$-groups, except for the ones covered by Corollary \ref{condition_on_G-prime}.

Now we show that in some more cases the Amit--Ashurst conjecture holds for groups of nilpotency class $2$. We achieve it through the following extension of Theorem \ref{word_as_chain}.

\begin{theorem}\label{exhaustive_p}
Let $w \in F_k$ be a word and $G$ be a $p$-group of nilpotency class $2$. 
\begin{enumerate}
\item[(i).] If $w \notin [F_k, F_k]$ then either $P_{w,G}$ is a constant function or $w$ and $v_p(p^r, \ell; s_1,\cdots,s_{n-1})$ are $F_k(G)$-automorphic for some $r > 0$; 
$s_{1} < s_{2} < \dots < s_{\ell-1} < r$, and
$s_j < r$ for $j \geq \ell$.

\item[(ii).] If $w \in [F_k, F_k]$ then for some $s_1, s_2, \dots, s_{n-1}$, the words $w$ and $v_p(0, 1; s_1,s_2,\cdots,s_{n-1})$ are $F_k(G)$-automorphic.
\end{enumerate}
\end{theorem}

\begin{proof}
If $w \in [F_k, F_k]$ then the theorem follows immediately from Lemma \ref{disjoint_commutators}.

Let $w \notin [F_k, F_k]$.
From Theorem \ref{word_as_chain}, and after applying elementary Nielsen transformations if required, we assume that 
\begin{align}\label{a-form-of-w}
w &= x_1^{p^r} \left( \prod_{i=1}^{\ell-1} [x_i,x_{i+1}]^{p^{s_{i}}} \right) \left( [x_{\ell},x_{\ell+1}]^{p^{s_{\ell}}} [x_{\ell+2},x_{\ell+3}]^{p^{s_{\ell+2}}} \dots [x_{n-1},x_n]^{p^{s_{n-1}}} \right)
\end{align}
where $1 \leq \ell < n \leq k$, and $r, s_i$ are nonnegative integers.
If possible, let $j \in \{1,2,\dots,\ell, \ell+2, \ell+4, n-1\}$ be the smallest index such that ${s_{j}} \geq r$. We consider the following two cases.\\

{\bfseries Case 1}. $j \in \{\ell-1,\ell,\ell+2,\dots,n-1 \}$.\\
In this case, we use the transformation $x_1 \mapsto x_1[x_{j+1},x_j]^{p^{s_{j}-r}}$ to eliminate the commutator $[x_j,x_{j+1}]^{p^{s_{j}}}$ from the expression of the word. A permutation of variables may then be carried out to eliminate free variables. This way, one indeed ensures the condition $s_j < r$ for 
$j \geq \ell-1$. \\

{\bfseries Case 2}. $j \in \{1,2,\dots,\ell-2 \}$; $\ell \geq 3$.\\
As in case 1, we first eliminate the commutator $[x_j,x_{j+1}]^{p^{s_{j}}}$ and obtain that $w$ is 
$F_k(G)$-automorphic to a word of the form
$$x_1^{p^r} \left( \prod_{i=1}^{j-1} [x_i,x_{i+1}]^{p^{s_{i}}} \right) \left( \prod_{i=j+1}^{\ell-1} [x_i,x_{i+1}]^{p^{s_{i}}} \right) \left( [x_{\ell},x_{\ell+1}]^{p^{s_{\ell}}} \dots [x_{n-1},x_n]^{p^{s_{n-1}}} \right).$$
and then apply Lemma \ref{disjoint_commutators} on $\left( \prod_{i=j+1}^{\ell-1} [x_i,x_{i+1}]^{p^{s_{i}}} \right)$ to split it into disjoint commutators. Further, using transformations as in case $1$ we obtain that $w$ is $F_k(G)$-automorphic to a word of the form
$$x_1^{p^r} \left( \prod_{i=1}^{j-1} [x_i,x_{i+1}]^{p^{s_{i}}} \right) \left( [x_{j+1},x_{j+2}]^{p^{s_{j+1}}} \dots [x_{n-1},x_n]^{p^{s_{n-1}}} \right),$$
where ${s_{t}} < r$ for every $t \geq j$. 
Thus in (\ref{a-form-of-w}), the assumption ${s_{j}} < r$, whenever $j \in \{1,2,\dots,\ell, \ell+2, \ell+4, \cdots, n-1\}$ can indeed be imposed, maintaining that the word of this form is indeed $F_k(G)$-automorphic to $w$.

Now, we examine the commutator $\left( \prod_{i=1}^{\ell-1} [x_i,x_{i+1}]^{p^{s_{i}}} \right)$ in (\ref{a-form-of-w}) along the lines of the proof of Lemma \ref{disjoint_commutators} to get the desired ordering $s_{1} < s_{2} < \dots < s_{\ell-1} < r$. Thus, if $r=0$ then $w$ is $F_k(G)$-automorphic to $x_1$, and if $r >0$ then $w$ is $F_k(G)$-automorphic to $v(p^r, \ell; s_1,\cdots,s_{n-1})$ with $s_{1} < s_{2} < \dots < s_{\ell-1} < r$, and $s_j < r$ for $j \geq \ell$.
\end{proof}

As a consequence of the above theorem, we obtain an elementary proof of Amit--Ashurst conjecture in two specific cases.

\begin{corollary}\label{power_word}
Let $G$ be a finite $p$-group of nilpotency class $2$. Let $r \in \mathbb{N}, \ell \leq k$, and $s_j \in \mathbb{N}$ be such that $s_j \geq r$ for every $j$. Then, for any word that is $F_k(G)$-automorphic to $v_p(r,\ell; s_1,s_2,\cdots,s_{n-1})$, the Amit--Ashurst conjecture holds on $G$. 
\end{corollary}

\begin{proof}
The proof follows from the claim that $v_p(r,\ell; s_1,s_2,\cdots,s_{n-1})$, with each $s_j \geq r$, is $F_k(G)$ automorphic to $x_1^{p^r}$. 
That these two words are $F_k(G)$-automorphic can be established using automorphisms as in the proof of Theorem \ref{exhaustive_p}.
\end{proof}

The following is an interesting corollary to Theorem \ref{exhaustive_p}.

\begin{corollary}\label{surjective_word}
Every surjective word on a finite nilpotent group of nilpotency class $2$ is $F_k(G)$-automorphic to the identity word and thus satisfies the Amit--Ashurst conjecture.
\end{corollary}
\begin{proof}
By \cite[Lemma 2.3]{CockeHoChirality}, we assume that $w={x_1}^ac$, where $c \in [F_k,F_k]$. Let $G=G_{p_1} \times G_{p_2} \times \dots \times G_{p_r}$ be the decomposition of $G$ in terms of its Sylow subgroups. Observe that the map induced by $w$ is surjective on $G$ iff it is surjective on each ${G_{p_i}}$. Further, if $p_i$ divides $a$ then $w=({x_1}^{\frac{a}{p_i}})^{p_i}c$ and thus $w(G_{p_i}) \subseteq G^{p_i}[G_{p_i},G_{p_i}]=\Phi(G_{p_i}) \neq G_{p_i}$. Here $\Phi(G_{p_i})$ denotes the Frattini subgroup of $G_{p_i}$.
This contradicts the surjectivity of $w$, and therefore no $p_i$ divides $a$. Now, by Theorem  \ref{exhaustive_p}, $w$ is $F_k(G)$-automorphic to a primitive word, whence $P_{w,G}(g)=\frac{1}{|G|}$ for every $g \in G$.
\end{proof}

We would like to exhibit the relevance of the above corollary from the perspective of the following theorem. 

\begin{theorem}{\bfseries\rm\cite[Theorem 1.1]{Pudur}~~}
Let $w \in F_k$ be a word that is identically distributed on $G$ to a primitive word for all finite groups $G$. Then, $w$ is a primitive word.
\end{theorem}

It is an easy consequence of Lemma \ref{aut_eq} that primitive words induce a surjective word map on every finite group $G$. However, not every surjective word is primitive. The Corollary \ref{surjective_word} may now be repackaged as the following.

\begin{corollary}\label{surjective-primitive}
Let $w \in F_k$ and $G$ be a finite nilpotent group of class $2$. Then, the following are equivalent.
\begin{enumerate}
\item[(i).]  $w$ is surjective on $G$.
\item[(ii).] $w$ is $F_k(G)$-automorphic to a primitive word.
\item[(iii).] $w$ is identically distributed on $G$ to a primitive word.
\end{enumerate}
\end{corollary}

We note that (i) $\Leftrightarrow$ (iii) is also contained in \cite[Theorem 8]{Cocke-Ho_19}, where the characterization of finite nilpotent groups is provided in terms of probability distributions of word maps. 
The question of equivalence (ii) $\Leftrightarrow$ (iii) is also posed in \cite[Question 2.3]{Amit-Vishne}.
We expect the Corollary \ref{surjective-primitive} to hold for all finite nilpotent groups. 
\bibliographystyle{amsalpha}
\bibliography{word-maps}

\providecommand{\bysame}{\leavevmode\hbox to3em{\hrulefill}\thinspace}
\providecommand{\MR}{\relax\ifhmode\unskip\space\fi MR }
\providecommand{\MRhref}[2]{%
  \href{http://www.ams.org/mathscinet-getitem?mr=#1}{#2}
}
\providecommand{\href}[2]{#2}
\begin{thebibliography}{KKK22}

\bibitem[AV11]{Amit-Vishne}
Alon Amit and Uzi Vishne, \emph{Characters and solutions to equations in finite
  groups}, J. Algebra Appl. \textbf{10} (2011), no.~4, 675--686. \MR{2834108}

\bibitem[CCT23]{CCT_2022}
Rachel~D. Camina, William~L. Cocke, and Anitha Thillaisundaram, \emph{The
  {A}mit-{A}shurst conjecture for finite metacyclic {$p$}-groups}, Eur. J.
  Math. \textbf{9} (2023), no.~3, 46. \MR{4604926}

\bibitem[CH18]{CockeHoChirality}
William Cocke and Meng-Che Ho, \emph{On the symmetry of images of word maps in
  groups}, Comm. Algebra \textbf{46} (2018), no.~2, 756--763. \MR{3764894}

\bibitem[CH19]{Cocke-Ho_19}
\bysame, \emph{The probability distribution of word maps on finite groups}, J.
  Algebra \textbf{518} (2019), 440--452.

\bibitem[CIT20]{CIT_2020}
Rachel~D. Camina, Ainhoa Iniguez, and Anitha Thillaisundaram, \emph{Word
  problems for finite nilpotent groups}, Archiv der Mathematik \textbf{115}
  (2020), no.~6, 599--609.

\bibitem[Coc18]{Cocke-Nilpotent}
William Cocke, \emph{Two characterizations of finite nilpotent groups}, J.
  Group Theory \textbf{21} (2018), no.~6, 1111--1116. \MR{3871476}

\bibitem[Cor98]{Renza-Cortini}
Renza Cortini, \emph{On special $p$-groups}, Bollettino dell'Unione Matematica
  Italiana (1998), no.~3, 677--689.

\bibitem[IS17]{IS_2017}
Ainhoa Iñiguez and Josu Sangroniz, \emph{Words and characters in finite
  p-groups}, Journal of Algebra \textbf{485} (2017), 230--246.

\bibitem[KKK22]{KKK_2022}
Dilpreet Kaur, Harish Kishnani, and Amit Kulshrestha, \emph{Word images and
  their impostors in finite nilpotent groups}, arXiv:2205.15369 (2022).

\bibitem[Lev11]{Levy_2011}
Matthew Levy, \emph{On the probability of satisfying a word in nilpotent groups
  of class 2}, arXiv:1101.4286 (2011).

\bibitem[Nie24]{Nielsen-24}
Jakob Nielsen, \emph{Die isomorphismengruppe der freien gruppen}, Mathematische
  Annalen \textbf{91} (1924), no.~3, 169--209.

\bibitem[PP15]{Pudur}
Doron Puder and Ori Parzanchevski, \emph{Measure preserving words are
  primitive}, J. Amer. Math. Soc. \textbf{28} (2015), no.~1, 63--97.
  \MR{3264763}

\end{thebibliography}

\end{document}